\documentclass[11pt,twoside, leqno]{article}

\usepackage{amssymb}
\usepackage{amsmath}
\usepackage{mathrsfs}
\usepackage{amsthm}
\usepackage{amsfonts}
\usepackage{enumerate}
\usepackage{esint}
\usepackage{latexsym,amssymb}
\usepackage{color}

\usepackage[misc,geometry]{ifsym}

\allowdisplaybreaks

\newtheorem{theorem}{Theorem}[section]

\newtheorem{corollary}[theorem]{Corollary}
\newtheorem{proposition}[theorem]{Proposition}
\theoremstyle{definition}
\newtheorem{remark}[theorem]{Remark}

\numberwithin{equation}{section}

\numberwithin{equation}{section}
\begin{document}

\arraycolsep=1pt

\title{\bf\Large
Optimal  Hardy inequalities associated with multipolar Schr\"{o}dinger operators
\footnotetext {\hspace{-0.35cm}
2010 {\it Mathematics Subject Classification}. Primary: 26D10;
Secondary: 42B37.
\endgraf {\it Key words and phrases}.
 Hardy  inequalities, Sharp constant, multipolar Schr\"{o}dinger operators.
}}

\author{Yongyang Jin, Li Tang, Can Ye, Shoufeng Shen$^{\textrm{\Letter}}$}
\date{}
\maketitle

\vspace{-0.6cm}

\begin{center}
\begin{minipage}{13.5cm}
{\small {\bf Abstract} \quad
We proved some optimal Hardy inequalities in $\mathbb{R}^{N}$ which is closely related to multipolar Schr\"{o}dinger operators with  mean-value type potentials,  these sharp inequalities  imply some multipolar type Heisenberg inequalities. We also obtained some improved  multipolar Hardy inequalities on bounded domains, moreover, we got the range of the best Hardy constant for a specific Hardy inequality.}

\end{minipage}
\end{center}

%
%\begin{abstract}
%In this paper, we obtained the Dunkl analogy of classical $L^{p}$ Hardy inequality for $p>N+2\gamma$, where $2\gamma$ is the degree of weight function associated with Dunkl operators, and $L^{p}$ Hardy inequalities with distant function in some G-invariant domains. Moreover we proved two Hardy-Rellich type inequalities for Dunkl operators.
%\end{abstract}
%
%
%
%\subjclass[2010]{26D10, 20F55, 42B37}
%\keywords{Hardy inequalities, Hardy-Rellich inequalities, Best constant, Dunkl operators.}
%
%\maketitle

\section{Introduction}
  A Hardy type inequality is said that there is a potential $V$ and a positive constant $\mu$ so that the following inequality
\begin{equation}\label{T-1-1}
\int_{\mathbb{R}^{N}}|\nabla u|^{2}dx\geq\mu\int_{\mathbb{R}^{N}}V|u|^{2}dx,
\end{equation}
holds. This issue is equivalent to study the positivity of Schr\"{o}dinger operators $-\Delta-\mu V$. Employing Sobolev embedding inequality $C_{N}\|u\|_{L^{2^*}(\mathbb{R}^{N})}^2\leq\|\nabla u\|_{L^{2}(\mathbb{R}^{N})}^2$(with sharp constant $C_{N}$), one obtain that $-\Delta-\mu V$ is nonnegative  if
\begin{equation}\label{T-1-2}
\|V\|_{L^{\frac{N}{2}}(\mathbb{R}^{N})}\leq\frac{C_{N}}{\mu}.
\end{equation}
See \cite{1} for more discussion of the potential energy operator $V$.

When $N\geq 3$, the well-known Hardy potentials $V=|x|^{-2}$, or so-called inverse square potential, does not satisfy (\ref{T-1-2}). In this case, we have the classical Hardy inequality with sharp constant
\begin{equation}\label{T-1-3}
\int_{\mathbb{R}^{N}}|\nabla u|^{2}dx\geq\frac{(N-2)^{2}}{4}\int_{\mathbb{R}^{N}}\frac{|u|^{2}}{|x|^{2}}dx.
\end{equation}

We mention that it is easy to see (\ref{T-1-3}) and Plancherel formula imply the Heisenberg inequality,
\begin{equation}\label{T-1-H}\int_{\mathbb{R}^{N}}|x|^{2}|u|^{2}dx\int_{\mathbb{R}^{N}}|\xi|^{2}|\hat{u}(\xi)|^{2}d\xi\geq \frac{(N-2)^{2}}{4},\end{equation}
where $\|u\|_{L^2{\mathbb{R}^{N}}}=1$ and $\hat{u}(\xi)=\frac{1}{(2\pi)^{n/2}}\int_{\mathbb{R}^{N}}e^{-i\xi\cdot x}u(x)dx$.
It's well known that (\ref{T-1-H}) is the beautiful mathematical description of the famous "Uncertainty Principle" in quantum mechanics.

There exists a great amount of literature on the generalization and improvement of (\ref{T-1-3}), see \cite{5, 7, 3, 6, 8, 9, 10, 4, 13, 11, 12} and the references therein.

There are also many works related to multipolar potentials $V(x)=\sum_{i=1}^{n}\frac{\mu_{i}}{|x-a_{i}|^2}$ with $n$ poles $a_{1},a_{2},\ldots,a_{n}$. These type of multipolar potentials are related to
the interaction of a finite number of electric dipoles. This form of systems are characterized by Hartree-Fock type model, which is the most commonly used model in Quantum Molecular \cite{14}. These
potentials are also applied in other fields such as combustion models and quantum cosmological models.

Consider the quadratic functional with respect to Schr\"{o}dinger operator $-\Delta-\sum_{i=1}^{n}\frac{\mu_{i}}{|x-a_{i}|^2}$,
\begin{equation*}
Q[u]:=\int_{\mathbb{R}^{N}}|\nabla u|^{2}dx-\sum_{i=1}^{n}\mu_{i}\int_{\mathbb{R}^{N}}\frac{|u|^{2}}{|x-a_{i}|^2}dx.
\end{equation*}
It is complicated to study the positivity of $Q[u]$ due to the relative position and interaction among the poles. The author in \cite{15} proved that Schr\"{o}dinger operator
$-\Delta-\sum_{i=1}^{n}\frac{\mu_{i}}{|x-a_{i}|^2}$ is positive if and only if $\sum_{i=1}^{n}\mu_{i}^{+}\leq\frac{(N-2)^{2}}{4}$(where $\mu_{i}^+=\max\{\mu_{i},0\}$) for any configuration of $a_{1},a_{2},\ldots,a_{n}$;
conversely, if $\sum_{i=1}^{n}\mu_{i}^{+}>\frac{(N-2)^{2}}{4}$, then there exist a configuration of $a_{1},a_{2},\ldots,a_{n}$ such that $Q[u]$ is not positive. These results then have been improved by authors in \cite{16}
that the existence of a configuration so that the quadratic form $Q[u]$ is positive is equivalent to $\mu_{i}\leq\frac{(N-2)^{2}}{4}$ for any $i=1,2,\ldots,n$ and $\sum_{i=1}^{n}\mu_{i}\leq\frac{(N-2)^{2}}{4}$. This
shows that the critical mass $\frac{(N-2)^{2}}{4}\frac{1}{|x-a_{i}|^2}$ for certain singular pole $a_{i}$ can be infinitely approximated, though all the other $\mu_{j}$ are small enough right now. Bosi, Dolbeault,
Esteban \cite{17} obtained a lower bound of the spectrum of the Schr\"{o}dinger operators $-\Delta-\mu\sum_{i=1}^{n}\frac{1}{|x-a_{i}|^2}$, $\mu\in(0,\frac{(N-2)^{2}}{4}]$, $n\geq 2$. In other words, consider
$\mu\in(0,\frac{(N-2)^{2}}{4}]$, $n\geq 2$, there exists a nonnegative constant $K_{n}\leq\pi^{2}$ such that
\begin{equation}\label{T-1-4}
\int_{\mathbb{R}^{N}}|\nabla u|^{2}dx+\frac{4K_{n}+4(n+1)\mu}{d^2}\int_{\mathbb{R}^{N}}|u|^{2}dx\geq\mu\sum_{i=1}^{n}\int_{\mathbb{R}^{N}}\frac{|u|^{2}}{|x-a_{i}|^2}dx,u\in H^{1}(\mathbb{R}^{N}),
\end{equation}
where $d:=\min\limits_{1\leq i\neq j\leq n}|a_{i}-a_{j}|$. Their proof depends on the well-known "IMS"  truncation method (see \cite{18,19}). Moreover, in an attempt to remove the lower order term, the author in \cite{17} obtained the following inequality for any $u\in H^{1}(\mathbb{R}^{N})$ and $n\geq 2$:
\begin{equation}\label{T-1-5}
\begin{split}
\int_{\mathbb{R}^{N}}|\nabla u|^{2}dx\geq & \frac{(N-2)^{2}}{4n}\sum_{i=1}^{n}\int_{\mathbb{R}^{N}}\frac{|u|^{2}}{|x-a_{i}|^2}dx \\
    & +\frac{(N-2)^{2}}{4n^2}\sum_{1\leq i<j\leq n}^{n}\int_{\mathbb{R}^{N}}\frac{|a_{i}-a_{j}|^2}{|x-a_{i}|^{2}|x-a_{j}|^{2}}|u|^{2}dx.
\end{split}
\end{equation}
When $x\rightarrow a_{i}$, the total mass near $a_{i}$ is $\frac{(N-2)^{2}}{4}\frac{2n-1}{n^2}\frac{1}{|x-a_{i}|^2}$, which is strictly smaller than $\frac{(N-2)^{2}}{4}\frac{1}{|x-a_{i}|^2}$.

The result above was improved by the authors in \cite{23} with an optimal weight. Specifically, when $n\geq2$, they proved the following inequality
\begin{equation}\label{T-1-6}
\int_{\mathbb{R}^{N}}|\nabla u|^{2}dx\geq\frac{(N-2)^{2}}{n^2}\sum_{1\leq i<j\leq n}\int_{\mathbb{R}^{N}}\frac{|a_{i}-a_{j}|^2}{|x-a_{i}|^{2}|x-a_{j}|^{2}}|u|^{2}dx,
\end{equation}
where the constant $\frac{(N-2)^{2}}{n^2}$ is sharp. This inequality provides a sharp positive singular quadratic potential tends gradually to
\begin{equation*}
\frac{(N-2)^{2}}{4}\frac{4n-4}{n^2}\frac{1}{|x-a_{i}|^2}
\end{equation*}
at any $a_{i},i=1,\ldots,n$, which is strictly larger than $\frac{(N-2)^{2}}{4}\frac{2n-1}{n^2}\frac{1}{|x-a_{i}|^2}$. So inequality (\ref{T-1-6}) can be seen as an improvement of (\ref{T-1-5}). By parallelogram rule in $\mathbb{R}^{N}$, inequality (\ref{T-1-6}) is equivalent to inequality
\begin{equation}\label{T-1-7}
\int_{\mathbb{R}^{N}}|\nabla u|^{2}dx\geq\frac{(N-2)^{2}}{n^2}\sum_{1\leq i<j\leq n}\int_{\mathbb{R}^{N}}\left|\frac{x-a_{i}}{|x-a_{i}|^{2}}-\frac{x-a_{j}}{|x-a_{j}|^{2}}\right|^{2}|u|^{2}dx.
\end{equation}
Later Devyver, Fraas and Pinchover in \cite{24} obtained another multipolar Hardy inequality for any $u\in H^{1}(\mathbb{R}^{N})$ reads as
\begin{equation}\label{T-1-8}
\int_{\mathbb{R}^{N}}|\nabla u|^{2}\geq \left(\frac{N-2)}{n+1}\right)^{2}\int_{\mathbb{R}^{N}}\left[\sum_{i=1}^{n}\frac{1}{|x-a_{i}|^2}+\sum_{1\leq i<j\leq n}^{n}\int_{\mathbb{R}^{N}}\frac{|a_{i}-a_{j}|^2}{|x-a_{i}|^{2}|x-a_{j}|^{2}}\right]|u|^{2}.
\end{equation}
The potential arise in (\ref{T-1-8}) is smaller than that in (\ref{T-1-6}) near every poles as it behaves asymptotically like
\begin{equation*}
\frac{(N-2)^{2}}{4}\frac{4n}{(n+1)^2}\frac{1}{|x-a_{i}|^2}.
\end{equation*}
However, authors in \cite{24} proved that the potential in (\ref{T-1-8}) is critical, i.e. inequality (\ref{T-1-8}) is impossible to be further improved. Actually, they also proved the criticality of the potential correlated with (\ref{T-1-6}).

We also mention some other results of multipolar Hardy inequalities: the authors in \cite{25} consider the inequality (\ref{T-1-6}) or (\ref{T-1-7}) in a domain $\Omega\subset\mathbb{R}^{N}$; the inequalities (\ref{T-1-6}) and (\ref{T-1-7})
are studied on Riemannion manifolds in \cite{26}(It is worth mentioning that the potentials in (\ref{T-1-6}) and (\ref{T-1-7}) are not equivalent anymore in general Riemannion manifolds); the authors in \cite{27} consider multipolar
Poincar\'{e}-Hardy inequalities on Cartan-Hadamard manifolds, which generalized the results in \cite{28} for single singularity.

Our goal in this paper is to consider mean-value type multipolar potentials so that the corresponding Hardy inequalities holds. This is motivated by noticing that the potential $\frac{1}{n}\sum_{i=1}^{n}\frac{1}{|x-a_{i}|^2}$ is just an
arithmetic mean of $n$ numbers $\frac{1}{|x-a_{i}|^2},i=1,2,\ldots,n$. In section 2 we list the main results of this paper. In section 3 we give the proof of Theorem 2.1 and 2.2. In the last section we obtain some improved  multipolar Hardy inequalities on bounded domains.

\section{Main results}
For  $n$ different points $a_{1},\ldots,a_{n}$ in $\mathbb{R}^{N}$, we denote that $d:=\min\limits_{1\leq i\neq j\leq n}|a_{i}-a_{j}|>0$. In this section we consider the following  mean-value type potentials
\begin{equation*}
V_{\lambda}(d_{1},d_{2},\ldots,d_{n}):=\left\{
\begin{aligned}
\left(\sum\limits_{i=1}\limits^{n}\alpha_{i}|x-a_{i}|^{-2\lambda}\right)^{\frac{1}{\lambda}}, & \lambda\in \mathbb{R}\setminus\{0\}, \\
\prod\limits_{i=1}\limits^{n}|x-a_{i}|^{-2\alpha_{i}}\hspace{9mm},  & \lambda=0,
\end{aligned}
\right.
\end{equation*}
where $\alpha_{i}\geq 0,i=1,2,\ldots,n$, $\sum\limits_{i=1}\limits^{n}\alpha_{i}=1$. We call $V_{-1}$ the  powered harmonic mean, $V_{0}$ powered geometric mean, $V_{1}$ powered arithmetic mean and $V_{2}$ powered quadratic mean respectively.
It is well-known that $V_{\lambda}$ is an increasing function  on $\lambda$, so we have the following inequalities
\begin{equation*}
\min_{1\leq i\leq n}|x-a_{i}|^{-2}\leq V_{-1} \leq V_{0}\leq V_{1}\leq V_{2}\leq\max_{1\leq i\leq n}|x-a_{i}|^{-2}.
\end{equation*}
Then we consider two potentials $V_{+\infty}$ and $V_{-\infty}$ in $\mathbb{R}^{N}$, where
\begin{equation*}
V_{+\infty}:=\max_{1\leq i\leq n}|x-a_{i}|^{-2},
\end{equation*}
\begin{equation*}
V_{-\infty}:=\min_{1\leq i\leq n}|x-a_{i}|^{-2}.
\end{equation*}
Our main results are as follow:
\begin{theorem}
We assert that the following multipolar Hardy  inequality holds for any $u\in H^{1}(\mathbb{R}^{N})$
\begin{equation}\label{*1}
\int_{\mathbb{R}^{N}}|\nabla u|^{2}dx\geq\frac{(N-2)^{2}}{4}\int_{\mathbb{R}^{N}}V_{+\infty}|u|^{2}dx,
\end{equation}
and the constant $\frac{(N-2)^{2}}{4}$ is sharp.
\end{theorem}
\begin{theorem}
We assert that the following multipolar Hardy  inequality holds for any $u\in H^{1}(\mathbb{R}^{N})$
\begin{equation}\label{*2}
\int_{\mathbb{R}^{N}}|\nabla u|^{2}dx\geq\frac{(N-2)^{2}}{4}\int_{\mathbb{R}^{N}}V_{-\infty}|u|^{2}dx,
\end{equation}
and the constant $\frac{(N-2)^{2}}{4}$ is sharp.
\end{theorem}
Recall that  $V_{-\infty}\leq V_{\lambda}\leq V_{+\infty}$, we get the following result from Theorem 2.1 and Theorem2.2.
\begin{corollary}
The following inequality holds
\begin{equation*}
\int_{\mathbb{R}^{N}}|\nabla u|^{2}dx\geq\frac{(N-2)^{2}}{4}\int_{\mathbb{R}^{N}}V_{\lambda}|u|^{2}dx, u\in H^{1}(\mathbb{R}^{N}),
\end{equation*}
moreover, the constant $\frac{(N-2)^{2}}{4}$ is sharp.
\end{corollary}
When $\lambda=1$ we have:
\begin{corollary}
The following multipolar Hardy inequality holds
\begin{equation}\label{*3}
\int_{\mathbb{R}^{N}}|\nabla u|^{2}dx\geq\frac{(N-2)^{2}}{4}\sum\limits_{i=1}\limits^{n}\int_{\mathbb{R}^{N}}\alpha_{i}\frac{|u|^{2}}{|x-a_{i}|^{2}}dx, u\in H^{1}(\mathbb{R}^{N}),
\end{equation}
the constant $\frac{(N-2)^{2}}{4}$ is sharp.
Especially, we have the following multipolar Hardy inequality
 \begin{equation}\label{}
\int_{\mathbb{R}^{N}}|\nabla u|^{2}dx\geq\frac{(N-2)^{2}}{4n}\sum\limits_{i=1}\limits^{n}\int_{\mathbb{R}^{N}}\frac{|u|^{2}}{|x-a_{i}|^{2}}dx, u\in H^{1}(\mathbb{R}^{N}),
\end{equation}
the constant $\frac{(N-2)^{2}}{4n}$ is sharp.
\end{corollary}
By Corollary 2.3 and H\"{o}lder inequality, we have the following multipolar type Heisenberg inequality.
\begin{corollary}
For any $\lambda\in \mathbb{R}$, the following inequality holds
\begin{equation*}
\int_{\mathbb{R}^{N}}|\nabla u|^{2}dx\int_{\mathbb{R}^{N}}\left(\sum\limits_{i=1}\limits^{n}\alpha_{i}|x-a_{i}|^{2\lambda}\right)^{\frac{1}{\lambda}}|u|^{2}dx\geq\frac{(N-2)^{2}}{4}\int_{\mathbb{R}^{N}}|u|^{2}dx, u\in H^{1}(\mathbb{R}^{N}).
\end{equation*}
Especially, we have
 \begin{equation}\label{}
\int_{\mathbb{R}^{N}}|\nabla u|^{2}dx\int_{\mathbb{R}^{N}}\left(\sum\limits_{i=1}\limits^{n}|x-a_{i}|^{2}\right)|u|^{2}dx\geq\frac{(N-2)^{2}}{4n}\int_{\mathbb{R}^{N}}|u|^{2}dx, u\in H^{1}(\mathbb{R}^{N}).
\end{equation}
\end{corollary}

\begin{proof}
In view of Corollary 2.3 we have
\begin{equation*}
\int_{\mathbb{R}^{N}}|\nabla u|^{2}dx\geq\frac{(N-2)^{2}}{4}\int_{\mathbb{R}^{N}}V_{-\lambda}|u|^{2}dx, u\in H^{1}(\mathbb{R}^{N}),
\end{equation*}
where $V_{-\lambda}^{-1}=\left(\sum\limits_{i=1}\limits^{n}\alpha_{i}|x-a_{i}|^{2\lambda}\right)^{\frac{1}{\lambda}}$. Thus
\begin{equation*}
\begin{split}
&\int_{\mathbb{R}^{N}}|\nabla u|^{2}dx\int_{\mathbb{R}^{N}}\left(\sum\limits_{i=1}\limits^{n}\alpha_{i}|x-a_{i}|^{2\lambda}\right)^{\frac{1}{\lambda}}|u|^{2}dx \\
\geq& \frac{(N-2)^{2}}{4}\int_{\mathbb{R}^{N}}V_{-\lambda}|u|^{2}dx\int_{\mathbb{R}^{N}}\left(\sum\limits_{i=1}\limits^{n}\alpha_{i}|x-a_{i}|^{2\lambda}\right)^{\frac{1}{\lambda}}|u|^{2}dx \\
\geq& \frac{(N-2)^{2}}{4}\int_{\mathbb{R}^{N}}|u|^{2}dx.
\end{split}
\end{equation*}
\end{proof}
\begin{remark}
When $n=1$, we recover the classical Heisenberg inequality (\ref{T-1-H}) by Corollary 2.5.
\end{remark}

Let $f(x)$ be a monotone function of one variable with an inverse $f^{-1}$. Define multipolar potentials as
\begin{equation*}
V_{f}(a_{1},a_{2},\ldots,a_{n}):=f^{-1}\left(\sum\limits_{i=1}\limits^{n}\alpha_{i}f(|x-a_{i}|^{-2})\right),
\end{equation*}
here $\alpha_{i}\geq 0,i=1,2,\ldots,n$, $\sum\limits_{i=1}\limits^{n}\alpha_{i}=1$. Then from Theorem 2.1 and 2.2 we affirm:
\begin{corollary}
We assert that
\begin{equation*}
\int_{\mathbb{R}^{N}}|\nabla u|^{2}dx\geq\frac{(N-2)^{2}}{4}\int_{\mathbb{R}^{N}}V_{f}|u|^{2}dx, u\in H^{1}(\mathbb{R}^{N}),
\end{equation*}
where the constant $\frac{(N-2)^{2}}{4}$ is sharp.
\end{corollary}

\section{Proof of Theorem 2.1 and 2.2}
Recall the Hardy type identity for $u\in C_{0}^{\infty}(\Omega)$ and $\varphi\in C^{2}(\Omega)$,
\begin{equation}\label{T-3-1}
\int_{\Omega}(|\nabla u|^{2}+\frac{\Delta\varphi}{\varphi}|u|^2)dx=\int_{\Omega}|\nabla u-\frac{\nabla\varphi}{\varphi}u|^{2}dx=\int_{\Omega}|\nabla(u\varphi^{-1})|^{2}\varphi^{2}dx.
\end{equation}
Equality (\ref{T-3-1}) leads to different Hardy type inequality along with  different choice of test function $\varphi$. In fact this equality can be more general, see \cite{21}, with the same assumption of $u$ and $\varphi$ ahead, $\alpha\in \mathbb{R}$, it holds,
\begin{equation*}
\int_{\Omega}|\nabla u|^{2}dx=\int_{\Omega}\left(\alpha(1-\alpha)\frac{|\nabla\varphi|^2}{|\varphi|^2}+\alpha\frac{\Delta\varphi}{\varphi}\right)dx+\int_{\Omega}|\nabla(u\varphi^{-1})|^{2}\varphi^{2\alpha}dx.
\end{equation*}
Due to the nonnegativity of the integral $\int_{\Omega}|\nabla(u\varphi^{-1})|^{2}\varphi^{2}dx$, we deduce Hardy inequality from (\ref{T-3-1}), namely,
\begin{equation}\label{T-3-2}
\int_{\Omega}|\nabla u|^{2}dx\geq\int_{\Omega}\frac{-\Delta\varphi}{\varphi}|u|^2dx.
\end{equation}
The difficulty is to find an appropriate function $\varphi$ to obtain the Hardy inequality we want. We also mention that (\ref{T-3-1}) also holds for distribution $\varphi$.

Potentials $V_{+\infty}$ and $V_{-\infty}$ are not in $C^{2}(\mathbb{R}^{N})$, but the set of non-differentiable points of these two potentials is contained in
$\tilde{T}:=T\bigcup\{a_{1},a_{2},\ldots,a_{n}\}$, $T:=\{x:\exists i,j\hspace{2mm}s.t.\hspace{2mm}|x-a_{i}|=|x-a_{j}|\}$. In fact the $\varphi$ we
would choose  are in $C^{2}(\mathbb{R}^{N}\setminus \tilde{T})\bigcap C(\mathbb{R}^{N})$, i.e.
\begin{equation}\label{T-3-2-1}
\varphi=\max_{1\leq i\leq n}|x-a_{i}|^{\frac{2-N}{2}}\hspace{2mm}or\hspace{2mm}\min_{1\leq i\leq n}|x-a_{i}|^{\frac{2-N}{2}}.
\end{equation}
Denote $E$ the set of non-differentiable points of $\varphi$, then $E\subseteqq \tilde{T}$. $T$ can
be written as
\begin{equation*}
T=\bigcup\limits_{1\leq i<j\leq n}T_{ij},
\end{equation*}
where $T_{ij}:=\{x:|x-a_{i}|=|x-a_{j}|\}$. $T_{ij}$ is a hyperplane so that its $N$ dimensional
Lebesgue measure is zero. Then $\tilde{T}$ is a zero measure set. $\varphi$ is in $C^{2}(\mathbb{R}^{N}\setminus \tilde{T})\bigcap C(\mathbb{R}^{N})$.
Thus we have the following identity  for $\varphi$ in (\ref{T-3-2-1}),
\begin{equation*}
\int_{\mathbb{R}^{N}\setminus \tilde{T}}|\nabla u|^{2}dx=\int_{\mathbb{R}^{N}\setminus
\tilde{T}}\frac{-\Delta\varphi}{\varphi}|u|^2dx+\int_{\mathbb{R}^{N}\setminus \tilde{T}}|\nabla(u\varphi^{-1})|^{2}\varphi^{2}dx.
\end{equation*}

\emph{Proof of Theorem 2.1.}

Let
\begin{equation*}
\varphi=\max_{1\leq i\leq n}|x-a_{i}|^{\frac{2-N}{2}}.
\end{equation*}

Then we consider a decomposition of $\mathbb{R}^{N}$ depending on the configuration of $\{a_{i}\}_{i=1}^{n}$. Define
\begin{equation}\label{T-3-E}
\begin{split}
E_{1}&=\{x\in\mathbb{R}^{N}\setminus\{a_{1},a_{2},\ldots,a_{n}\}:\varphi(x)=|x-a_{i}|^{\frac{2-N}{2}}\}, \\
\vdots&  \\
E_{i}&=\{x\in\mathbb{R}^{N}\setminus\{a_{1},a_{2},\ldots,a_{n}\}\setminus\bigcup_{k=1}^{i-1}E_{k}:\varphi(x)=|x-a_{i}|^{\frac{2-N}{2}}\},i=2,\ldots,n.
\end{split}
\end{equation}
It is obvious that $E_{i}$ verify two properties:
\begin{equation*}
E_{i}\bigcap E_{j}=\emptyset, i\neq j;
\end{equation*}
\begin{equation*}
\bigcup_{i=1}^{n}E_{i}=\mathbb{R}^{N}\setminus\{a_{1},a_{2},\ldots,a_{n}\}.
\end{equation*}
For every $x\in E_{i}^{\circ}$,
\begin{equation*}
\frac{-\Delta\varphi}{\varphi}=\frac{(N-2)^2}{4}\frac{1}{|x-a_{i}|^2}.
\end{equation*}
Note that $\varphi=V_{+\infty}^{\frac{N-2}{4}}$, and $\frac{N-2}{4}>0$ when $N\geq3$. Thus
\begin{equation*}
\frac{-\Delta\varphi}{\varphi}=\frac{(N-2)^2}{4}V_{+\infty},\hspace{2mm}in\hspace{2mm}\mathbb{R}^{N}\setminus \tilde{T}.
\end{equation*}
Thus we deduce inequality (\ref{*1}) holds since $\tilde{T}$ is a zero measure set. Moreover we have
\begin{equation}\label{T-3-3}
\int_{\mathbb{R}^{N}}|\nabla u|^{2}dx-\frac{(N-2)^{2}}{4}\int_{\mathbb{R}^{N}}V_{+\infty}|u|^{2}dx=\int_{\mathbb{R}^{N}}|\nabla(u\varphi^{-1})|^{2}\varphi^{2}dx.
\end{equation}
The gradient in the r.h.s. of (\ref{T-3-3}) is in the sense of weak derivative.

Next we prove the optimality of $\frac{(N-2)^2}{4}$. For $\forall x\in B(a_{i},\frac{d}{2}):=\{x:|x-a_{i}|<\frac{d}{2}\}$, and any $j\neq i$,
\begin{equation*}
|x-a_{j}|\geq|a_{i}-a_{j}|-|x-a_{i}|>\frac{d}{2}\geq|x-a_{i}|,
\end{equation*}
so $B(a_{i},\frac{d}{2})\subseteq E_{i}$ for any $1\leq i\leq n$. Now for this representation we can also define a series of cut-off functions as follow
\begin{equation*}
\psi_{\varepsilon,i}=\left\{
\begin{aligned}
0\hspace{7mm}, &x\in B(a_{i},\varepsilon^{2})\bigcup \mathbb{R}^{N}\setminus E_{i}, \\
\frac{log\frac{|x-a_{i}|}{\varepsilon^2}}{log\frac{1}{\varepsilon}}, &x\in B(a_{i},\varepsilon)\setminus B(a_{i},\varepsilon^{2}), \\
1\hspace{7mm}, &x\in B(a_{i},\varepsilon)^{c}\bigcap E_{i}.
\end{aligned}
\right.
\end{equation*}
Here  $\varepsilon>0$ is small enough. Then we consider  $u_{\varepsilon}=\sum_{i=1}^{n}u_{\varepsilon,i}$,
where
\begin{equation*}
u_{\varepsilon,i}=\psi_{\varepsilon,i}|x-a_{i}|^{\frac{2-N}{2}-\varepsilon}.
\end{equation*}
Take $u_{\varepsilon}$ into (\ref{T-3-3}).
Firstly,
\begin{equation}\label{T-3-4}
\begin{split}
 & \int_{\mathbb{R}^{N}}|\nabla(u_{\varepsilon}\varphi^{-1})|^{2}\varphi^{2}dx \\
 =& \sum_{i=1}^{n}\int_{E_{i}}|\nabla(u_{\varepsilon}\varphi^{-1})|^{2}\varphi^{2}dx \\
 =& \sum_{i=1}^{n}\int_{E_{i}}|\nabla(\psi_{\varepsilon,i}|x-a_{i}|^{-\varepsilon})|^{2}|x-a_{i}|^{2-N}dx \\
 \leq& 2\sum_{i=1}^{n}\left(\int_{B(a_{i},\varepsilon)\setminus B(a_{i},\varepsilon^{2})}\frac{1}{log\frac{1}{\varepsilon}}|x-a_{i}|^{-2\varepsilon-N}dx+\varepsilon^2 \int_{B(a_{i},\varepsilon^2)^{c}}|x-a_{i}|^{-2\varepsilon-N}dx\right) \\
 =&2n\omega_{N}\left(\frac{1}{log\frac{1}{\varepsilon}}\int_{\varepsilon^2}^{\varepsilon}r^{-2\varepsilon-1}dr+\varepsilon^{2}\int_{\varepsilon^2}^{+\infty}r^{-2\varepsilon-1}dr\right) \\
 =& 2n\omega_{N}\left(\frac{\varepsilon^{-4\varepsilon}-\varepsilon^{-2\varepsilon}}{2\varepsilon log\frac{1}{\varepsilon}}+\frac{1}{2}\varepsilon^{1-4\varepsilon}\right)\rightarrow 2n\omega_{N}, as\hspace{2mm}{\varepsilon\rightarrow 0}.
\end{split}
\end{equation}
Then we know that for any $1\leq i\leq n$, $E_{i}$ contains a ball with radius $\frac{d}{2}$, so
\begin{equation}\label{T-3-5}
\begin{split}
\int_{\mathbb{R}^{N}}V_{+\infty}|u_{\varepsilon}|^{2}dx & =\sum_{i=1}^{n}\int_{E_{i}}|x-a_{i}|^{-2\varepsilon-N}|\psi_{\varepsilon,i}|^{2}dx \\
& \geq\sum_{i=1}^{n}\int_{B(a_{i},\frac{d}{2})\setminus B(a_{i},\varepsilon)}|x-a_{i}|^{-2\varepsilon-N}dx \\
& \geq n\omega_{N}\int_{\varepsilon}^{\frac{d}{2}}r^{-2\varepsilon-1}dr \\
& =n\omega_{N}\frac{\varepsilon^{-2\varepsilon}-\left(\frac{d}{2}\right)^{-2\varepsilon}}{2\varepsilon}\rightarrow +\infty, as\hspace{2mm}\varepsilon\rightarrow 0.
\end{split}
\end{equation}
Combining (\ref{T-3-4}) and (\ref{T-3-5}) we have
\begin{equation*}
\begin{split}
 & \lim_{\varepsilon\rightarrow 0}\frac{\int_{\mathbb{R}^{N}}|\nabla u_{\varepsilon}|^{2}dx}{\int_{\mathbb{R}^{N}}V_{+\infty}|U_{\varepsilon}|^{2}dx} \\
 = & \lim_{\varepsilon\rightarrow0}\left(\frac{(N-2)^{2}}{4}+\frac{\int_{\mathbb{R}^{N}}|\nabla(u_{\varepsilon}\varphi^{-1})|^{2}\varphi^{2}dx}{\int_{\mathbb{R}^{N}}V_{+\infty}|U_{\varepsilon}|^{2}dx}\right) \\
 = & \frac{(N-2)^{2}}{4}.
\end{split}
\end{equation*}
Thus we complete the proof of Theorem 2.1.

\emph{Proof of theorem 2.2.}
Let
\begin{equation*}
\varphi=\min_{1\leq i\leq n}|x-a_{i}|^{\frac{2-N}{2}}.
\end{equation*}
By similar argument we have the following equality for a.e. $x\in \mathbb{R}^{N}$,
\begin{equation*}
\frac{-\Delta\varphi}{\varphi}=\frac{(N-2)^2}{4}V_{-\infty},
\end{equation*}
Thus  inequality (\ref{*2}) holds, and
\begin{equation*}
\int_{\mathbb{R}^{N}}|\nabla u|^{2}dx-\frac{(N-2)^{2}}{4}\int_{\mathbb{R}^{N}}V_{-\infty}|u|^{2}dx=\int_{\mathbb{R}^{N}}|\nabla(u\varphi^{-1})|^{2}\varphi^{2}dx.
\end{equation*}

It remains to prove the sharpness of the constant. For any $\varepsilon>0$, Let
\begin{equation*}
u_{\varepsilon}=\min_{1\leq i\leq n}|x-a_{i}|^{\frac{2-N}{2}-\varepsilon},
\end{equation*}
when $n=1$, the optimality of $\frac{(N-2)^2}{4}$ has already known. When $n\geq 2$, $u_{\varepsilon}$ belongs to $D^{1,2}(\mathbb{R}^{N})$ with the norm
\begin{equation*}
\|u\|_{D^{1,2}(\mathbb{R}^{N})}=\langle\nabla u,\nabla u\rangle,
\end{equation*}
define $\tilde{E}_{i}$ just as in (\ref{T-3-E}) by taking $\varphi(x)=\min_{1\leq i\leq n}|x-a_{i}|^{\frac{2-N}{2}}$, note that $\tilde{E}_{i}\subseteq B(a_{i},\frac{d}{2})^{c}$ for any $i=1,2,\ldots,n$,
we obtain by direct computation,
\begin{equation*}
\lim_{\varepsilon\rightarrow 0}\frac{\int_{\mathbb{R}^{N}}|\nabla u_{\varepsilon}|^{2}dx}{\int_{\mathbb{R}^{N}}V_{-\infty}|U_{\varepsilon}|^{2}dx}=\lim_{\varepsilon\rightarrow 0}
\frac{(\frac{N-2}{2}+\varepsilon)^{2}\sum\limits_{i=1}\limits^{n}\int_{\tilde{E}_{i}}|x-a_{i}|^{-2\varepsilon-N}dx}{\sum\limits_{i=1}\limits^{n}\int_{\tilde{E}_{i}}|x-a_{i}|^{-2\varepsilon-N}dx}=\frac{(N-2)^2}{4}.
\end{equation*}

Our results recover that the result in \cite{15} that Schr\"{o}dinger operator $-\Delta-\frac{(N-2)^{2}}{4}V_{1}=-\Delta-\frac{(N-2)^{2}}{4}\sum_{i=1}^{n}\frac{\alpha_{i}}{|x-a_{i}|^2}$ is positive, and the constant $\frac{(N-2)^{2}}{4}$ cannot be larger, i.e. (\ref{*3}) is a sharp Hardy type inequality. But in fact we can add a positive term in the r.h.s. of (\ref{*3}). Actually, in (\ref{T-3-1}) let
\begin{equation*}
\varphi_{1}=\prod\limits_{i=1}^{n}|x-a_{i}|^{\beta\alpha_{i}},\hspace{4mm}\varphi_{2}=\sum_{i=1}^{n}\alpha_{i}|x-a_{i}|^{\frac{2-N}{2}}.
\end{equation*}
Compute directly we have
\begin{equation}\label{T-3-6}
\begin{split}
 \int_{\mathbb{R}^{N}}|\nabla(u\varphi_{1}^{-1})|^{2}\varphi_{1}^{2}dx= & \int_{\mathbb{R}^{N}}|\nabla u|^{2}dx+[\beta^{2}+\beta(N-2)]\sum_{i=1}^{n}\int_{\mathbb{R}^{N}}\alpha_{i}\frac{|u|^{2}}{|x-a_{i}|^2}dx \\
     & -\alpha^{2}\sum_{1\leq i<j\leq n}^{n}\int_{\mathbb{R}^{N}}\alpha_{i}\alpha_{j}\frac{\left|a_{i}-a_{j}\right|^{2}}{|x-a_{i}|^{2}|x-a_{j}|^{2}}|u|^{2}dx,
\end{split}
\end{equation}
and
\begin{equation*}
\begin{split}
&\int_{\mathbb{R}^{N}}|\nabla(u\varphi_{2}^{-1})|^{2}\varphi_{2}^{2}dx=\int_{\mathbb{R}^{N}}|\nabla u|^{2}dx-\frac{(N-2)^{2}}{4}\sum\limits_{i=1}^{n}\int_{\mathbb{R}^{N}}\alpha_{i}\frac{|u|^{2}}{|x-a_{i}|^2}dx \\
     & -\frac{(N-2)^{2}}{4}\int_{\mathbb{R}^{N}}\frac{\sum\limits_{1\leq i<j\leq n}^{n}\alpha_{i}\alpha_{j}(|x-a_{i}|^{-2}-|x-a_{j}|^{-2})\left(|x-a_{i}|^{\frac{2-N}{2}}-|x-a_{j}|^{\frac{2-N}{2}}\right)}{\sum\limits_{i=1}^{n}\alpha_{i}|x-a_{i}|^{\frac{2-N}{2}}}|u|^{2}dx.
\end{split}
\end{equation*}
The term $(|x-a_{i}|^{-2}-|x-a_{j}|^{-2})\left(|x-a_{i}|^{\frac{2-N}{2}}-|x-a_{j}|^{\frac{2-N}{2}}\right)\geq 0$ when $N\geq3$. Thus we obtain the generalization of inequalities (\ref{T-1-5}) and (\ref{T-1-6}) by letting $\beta=\frac{2-N}{2}$ and $\beta=2-N$ respectively in (\ref{T-3-6}), also an improvement of inequality (\ref{*3}).
\begin{theorem}
The following inequality holds for any $u\in H^{1}(\mathbb{R}^{N})$
\begin{equation}\label{T-3-7}
\begin{split}
\int_{\mathbb{R}^{N}}|\nabla u|^{2}dx \geq& \frac{(N-2)^{2}}{4}\sum_{i=1}^{n}\int_{\mathbb{R}^{N}}\alpha_{i}\frac{|u|^{2}}{|x-a_{i}|^2}dx \\
     & +\frac{(N-2)^{2}}{4}\sum_{1\leq i<j\leq n}^{n}\int_{\mathbb{R}^{N}}\alpha_{i}\alpha_{j}\frac{\left|a_{i}-a_{j}\right|^{2}}{|x-a_{i}|^{2}|x-a_{j}|^{2}}|u|^{2}dx.
\end{split}
\end{equation}
\end{theorem}

\begin{theorem}
There holds
\begin{equation*}
\int_{\mathbb{R}^{N}}|\nabla u|^{2}dx \geq(N-2)^{2}\sum_{1\leq i<j\leq n}^{n}\int_{\mathbb{R}^{N}}\alpha_{i}\alpha_{j}\frac{\left|a_{i}-a_{j}\right|^{2}}{|x-a_{i}|^{2}|x-a_{j}|^{2}}|u|^{2}dx, \hspace{2mm}u\in H^{1}(\mathbb{R}^{N}).
\end{equation*}
\end{theorem}

\begin{theorem}
For any $u\in H^{1}(\mathbb{R}^{N})$ there holds
\begin{equation*}
\begin{split}
&\int_{\mathbb{R}^{N}}|\nabla u|^{2}dx\geq\frac{(N-2)^{2}}{4}\sum\limits_{i=1}^{n}\int_{\mathbb{R}^{N}}\alpha_{i}\frac{|u|^{2}}{|x-a_{i}|^2}dx \\
     & +\frac{(N-2)^{2}}{4}\int_{\mathbb{R}^{N}}\frac{\sum\limits_{1\leq i<j\leq n}^{n}\alpha_{i}\alpha_{j}(|x-a_{i}|^{-2}-|x-a_{j}|^{-2})\left(|x-a_{i}|^{\frac{2-N}{2}}-|x-a_{j}|^{\frac{2-N}{2}}\right)}{\sum\limits_{i=1}^{n}\alpha_{i}|x-a_{i}|^{\frac{2-N}{2}}}|u|^{2}dx.
\end{split}
\end{equation*}
\end{theorem}
Inequality (\ref{T-3-7}) does not break the optimality of constant $\frac{(N-2)^{2}}{4}$ in (\ref{*3}) because the potential $\sum_{1\leq i<j\leq n}^{n}\frac{|a_{i}-a_{j}|^2}{|x-a_{i}|^{2}|x-a_{j}|^{2}}$ cannot be compared with $V_{1}$ near infinity. Actually it behaves asymptotically like
\begin{equation*}
\sum_{1\leq i<j\leq n}^{n}\frac{|a_{i}-a_{j}|^2}{|x-a_{i}|^{2}|x-a_{j}|^{2}}\sim O\left(\frac{1}{|x|^4}\right),|x|\rightarrow\infty.
\end{equation*}

\section{Some improvements on bounded domains}

The classical Hardy inequality which corresponds to  $V=|x|^{-2}$ and $\mu=\frac{(N-2)^2}{4}$ in (\ref{T-1-1})  is
\begin{equation}\label{T-4-1}
\int_{\Omega}|\nabla u|^{2}dx\geq\frac{(N-2)^2}{4}\int_{\Omega}\frac{|u|^{2}}{|x|^{2}}dx, \hspace{5mm}\forall u\in H_{0}^{1}(\Omega)
\end{equation}
where  $N\geq 3$ and $\Omega$ is an open subset of $\mathbb{R}^{N}$ containing the origin. The constant is optimal and never achieved. When $\Omega=\mathbb{R}^{N}$, it is impossible to add a strictly positive term in the r.h.s. of (\ref{T-4-1}).
But if $\Omega$ is bounded, Brezis and V\'{a}zquez firstly in \cite{3} obtained an improvement of (\ref{T-4-1}), the so-called Hardy-Poincar\'{e} inequality
\begin{equation}\label{T-4-2}
\int_{\Omega}|\nabla u|^{2}dx\geq\frac{(N-2)^2}{4}\int_{\Omega}\frac{|u|^{2}}{|x|^{2}}dx+\frac{h_{2}}{R_{\Omega}^2}\int_{\Omega}|u|^{2}dx, \hspace{5mm}\forall u\in H_{0}^{1}(\Omega)
\end{equation}
where $R_{\Omega}=\left(\frac{|\Omega|}{\omega_{N}}\right)^\frac{1}{n}$, $\omega_{N}$ is the volume of N-dimensional unit ball; $h_{2}$ is the first eigenvalue of Laplace operator in the unit disk of $\mathbb{R}^{2}$. In addition, they proved that when $\Omega$ is a ball, the constant $\frac{h_{2}}{R_{\Omega}^2}$
is  sharp and never attained. They also obtained another improvement. When $N\geq 3$ and $1<q<2^{*}=\frac{2N}{N-2}$, then for any $u\in H_{0}^{1}(\Omega)$,
\begin{equation}\label{T-4-3}
\int_{\Omega}|\nabla u|^{2}dx\geq\frac{(N-2)^2}{4}\int_{\Omega}\frac{|u|^{2}}{|x|^{2}}dx+C(\Omega)\left(\int_{\Omega}|u|^{q}dx\right)^{\frac{2}{q}}.
\end{equation}

Motivated by (\ref{T-4-2}) and (\ref{T-4-3}), we have the following two similar improvements in the case of multiple singularities.
\begin{theorem}
Let $\Omega\subset\mathbb{R}^{N}(N\geq 3)$ be a bounded domain, $a_{i} (i=1,2,...,n)$ be $n$ different points in $\Omega$. There holds
\begin{equation*}
\int_{\Omega}|\nabla u|^{2}dx\geq\frac{(N-2)^{2}}{4}\int_{\Omega}V_{+\infty}|u|^{2}dx+\frac{1}{n}\frac{h_{2}}{R_{\Omega}^2}\int_{\Omega}|u|^{2}dx, \hspace{4mm}\forall u\in H_{0}^{1}(\Omega).
\end{equation*}
\end{theorem}
\begin{theorem}
Let $\Omega\subset\mathbb{R}^{N}(N\geq 3)$ be a bounded domain, $a_{i} (i=1,2,...,n)$ be $n$ different points in $\Omega$, $1<q<2^{*}$. There exists a positive constant $C(q,\Omega)$ such that the following inequality holds for any $u\in H_{0}^{1}(\Omega)$
\begin{equation*}
\int_{\Omega}|\nabla u|^{2}dx\geq\frac{(N-2)^{2}}{4}\int_{\Omega}V_{+\infty}|u|^{2}dx+\frac{1}{n}C(q,\Omega)\left(\int_{\Omega}|u|^{q}dx\right)^{\frac{2}{q}}.
\end{equation*}
\end{theorem}
The proof of Theorem 4.1 is similar to that of Theorem 4.2. We only prove Theorem 4.2 here. \\
\textbf{\emph{Proof of Theorem 4.2.}}\\
Let $u=v\max\limits_{1\leq i\leq n}|x-a_{i}|^{\frac{2-N}{2}}$, $u_{i}=v|x-a_{i}|^{\frac{2-N}{2}}$, $supp v\subseteq\Omega$. We observe that $u_{i}=u$ in $E_{i}\bigcap\Omega$. Thus using (\ref{T-3-1}),
\begin{equation*}
\begin{split}
   &  \int_{\Omega}|\nabla u|^{2}dx-\frac{(N-2)^{2}}{4}\sum_{i=1}^{n}\int_{\Omega}V_{+\infty}|u|^{2}dx \\
= &\int_{\Omega}|\nabla v|^{2}\max_{1\leq i\leq n}|x-a_{i}|^{2-N}dx \\
 \geq & \frac{1}{n}\sum_{i=1}^{n}\int_{\Omega}|\nabla v|^{2}|x-a_{i}|^{2-N}dx \\
 = & \frac{1}{n}\sum_{i=1}^{n}\left(\int_{\Omega}|\nabla u_{i}|^{2}dx-\frac{(N-2)^{2}}{4}\int_{\Omega}\frac{|u_{i}|^{2}}{|x-a_{i}|^2}dx\right) \\
 \geq & \frac{1}{n}\sum_{i=1}^{n}C(\Omega)\left(\int_{\Omega}|u_{i}|^{q}dx\right)^{\frac{2}{q}} \\
 \geq & \frac{1}{n}\sum_{i=1}^{n}C(\Omega)\left(\int_{E_{i}\bigcap\Omega}|u|^{q}dx\right)^{\frac{2}{q}} \\
 \geq & \frac{1}{n}C(q,\Omega)\left(\int_{\Omega}|u|^{q}dx\right)^{\frac{2}{q}}.
\end{split}
\end{equation*}
We complete the proof.

The constant in (\ref{T-4-1}) is optimal in bounded domains containing the origin. However, if $a_{i}\in\Omega,i=1,\ldots,n$, $n\geq 2$, and $\Omega$ is a bounded open subset of $\mathbb{R}^{N}$. Then from (\ref{T-3-7}) we obtain
\begin{equation*}
\int_{\Omega}|\nabla u|^{2}dx\geq\frac{(N-2)^{2}}{4}\sum_{i=1}^{n}\int_{\Omega}\tilde{\alpha_{i}}\frac{|u|^{2}}{|x-a_{i}|^2}dx,
\end{equation*}
where $\tilde{\alpha_{i}}=\alpha_{i}+\frac{1}{2}\sum\limits_{j=1,j\neq i}\limits^{n}\alpha_{i}\alpha_{j}\frac{|a_{i}-a_{j}|^2}{|diam_{\Omega}|^2}$, $\alpha_{i}\geq 0,i=1,2,\ldots,n$, $\sum\limits_{i=1}\limits^{n}\alpha_{i}=1$, $diam_{\Omega}$ denotes the diameter of $\Omega$.
When $n\geq2$, $\frac{(N-2)^{2}}{4}\sum_{i=1}^{n}\tilde{\alpha_{i}}$ is strictly larger than $\frac{(N-2)^{2}}{4}$. Thus the constant $\frac{(N-2)^{2}}{4}$ in (\ref{*3}) is not optimal.
We aim to find a better potential in a bounded domain or deduce  the range the optimal constant. Motivated by \cite{25}, we obtain the following result.
\begin{theorem}
Let $N\geq 3$ and $\Omega\subset\mathbb{R}^{N}$ be a bounded domain with $n$ different poles $a_{1},\ldots,a_{n}\in\Omega$, $n\geq 2$. Given $\gamma_{i}>0$, $i=1,\ldots,n$.
If the following sharp Hardy inequality
\begin{equation*}
\int_{\Omega}|\nabla u|^{2}dx\geq C^{*}(\Omega)\sum_{i=1}^{n}\int_{\Omega}\gamma_{i}\frac{|u|^{2}}{|x-a_{i}|^2}dx,
\end{equation*}
holds for any $u\in H_{0}^{1}(\Omega)$, then we have
\begin{equation}\label{T-4-4}
C^{*}(\Omega)\sum_{i=1}^{n}\gamma_{i}>\frac{(N-2)^{2}}{4}\hspace{10mm}and\hspace{10mm}\max\limits_{1\leq i\leq n}{C^{*}(\Omega)\gamma_{i}}\leq\frac{(N-2)^{2}}{4}.
\end{equation}
\end{theorem}
\begin{proof}
The first inequality of (\ref{T-4-4}) is obtained by the above discussion. The second inequality could be deduced by Hardy inequality (\ref{T-4-1}). Assume there is a $\gamma_{k}$
such that $C^{*}(\Omega)\gamma_{k}>\frac{(N-2)^{2}}{4}$. Choose a ball $B(a_{k},\epsilon)$,
$\epsilon$ small enough such that $B(a_{k},\epsilon)\subset\Omega$, and
\begin{equation*}
C^{*}(\Omega)\sum_{i=1}^{n}\gamma_{i}\frac{1}{|x-a_{i}|^2}=C^{*}(\Omega)\gamma_{k}\frac{1}{|x-a_{k}|^2}(1+o(1)).
\end{equation*}
Thus for any $u\in C_{0}^{\infty}(B(a_{k},\epsilon))$,
\begin{equation*}
\int_{\Omega}|\nabla u|^{2}dx\geq C^{*}(\Omega)\gamma_{k}(1+o(1))\int_{\Omega}\frac{|u|^{2}}{|x-a_{k}|^2}dx>\frac{(N-2)^{2}}{4}\int_{\Omega}\frac{|u|^{2}}{|x-a_{k}|^2}dx.
\end{equation*}
This is  contradicted with (\ref{T-4-1}). We complete the proof.
\end{proof}
The following proposition reveals that there exists subset $U$ of $\Omega$ which contain all the poles such that the Hardy constant  of the multipolar Hardy potentials $\sum_{i=1}^{n}\frac{1}{|x-a_{i}|^2}$ in $U$ can be close to $\frac{(N-2)^{2}}{4}$ infinitely.

\begin{proposition}
Let $N\geq 3$ and $\Omega\subset\mathbb{R}^{N}$ be a bounded domain with $n$ poles $a_{1},\ldots,a_{n}$, $n\geq 2$, and $V_{*}=\sum_{i=1}^{n}\frac{1}{|x-a_{i}|^2}$. Then for any $\epsilon>0$,
there exists a domain $U_{\epsilon}\subset\Omega$ such that the following inequality holds for any $u\in C_{0}^{\infty}(U_{\epsilon})$,
\begin{equation*}
\int_{U_{\epsilon}}|\nabla u|^{2}dx\geq\left(\frac{(N-2)^{2}}{4}-\epsilon\right)\int_{U_{\epsilon}}V_{*}|u|^{2}dx.
\end{equation*}
\end{proposition}

\begin{proof}
We take  $U_{\epsilon}=\bigcup\limits_{i=1}^{n}B(a_{i},r_{\epsilon})\subset\Omega$, where $r_{\epsilon}$ small enough so that $B(a_{i},r_{\epsilon})\bigcap B(a_{j},r_{\epsilon})=\emptyset$ for any $i\neq j$. From (\ref{*1}) we have
\begin{equation}\label{T-4-5}
\int_{\Omega}|\nabla u|^{2}dx\geq\frac{(N-2)^{2}}{4}\int_{\Omega}W_{1}V_{*}|u|^{2}dx,
\end{equation}
where $W_{1}=V_{+\infty}V_{*}$. In view of the behavior of $V_{+\infty}$ and $V_{*}$ near each pole we have
\begin{equation*}
\lim\limits_{x\rightarrow a_{i}}W_{1}(x)=1, i=1,\ldots,n.
\end{equation*}

Then for any $\epsilon>0$, we choose $r_{\epsilon}$ small enough, so that
\begin{equation}\label{T-4-6}
|W_{1}(x)-1|<\delta(\epsilon),\hspace{2mm}\forall x\in B(a_{i},r_{\epsilon}),\hspace{2mm}i=1,\ldots,n,
\end{equation}
where $\delta(\epsilon)=\frac{4\epsilon}{(N-2)^{2}}$. Since $U_{\epsilon}$ is composed of $n$ connected branch, for any $u\in C_{0}^{\infty}(U_{\epsilon})$, we can denote $u=\sum_{i=1}^{n}u_{i}$, here $u_{i}\in C_{0}^{\infty}(B(a_{i},r_{\epsilon}))$. Then, combining (\ref{T-4-5}) and (\ref{T-4-6}) we have
\begin{equation*}
\begin{split}
\int_{\Omega}|\nabla u|^{2}dx= & \sum_{i=1}^{n}\int_{B(a_{i},r_{\epsilon})}|\nabla u_{i}|^{2}dx \\
   \geq & \sum_{i=1}^{n}\frac{(N-2)^{2}}{4}\int_{B(a_{i},r_{\epsilon})}W_{1}V_{*}|u|^{2}dx \\
   \geq & \sum_{i=1}^{n}\left(\frac{(N-2)^{2}}{4}(1-\delta(\epsilon))\right)\int_{B(a_{i},r_{\epsilon})}V_{*}|u|^{2}dx\\
   = & \sum_{i=1}^{n}\left(\frac{(N-2)^{2}}{4}-\epsilon\right)\int_{B(a_{i},r_{\epsilon})}V_{*}|u_{i}|^{2}dx  \\
   = & \left(\frac{(N-2)^{2}}{4}-\epsilon\right)\int_{\Omega}V_{*}|u|^{2}dx.
\end{split}
\end{equation*}

The proof of Proposition 4.4 is completed.
\end{proof}

We end this paper by concluding a problem presented in  \cite{25}.
\begin{corollary}
Let $N\geq 3$ and $\Omega\subset\mathbb{R}^{N}$ is a bounded domain with $a_{1},\ldots,a_{n}\in\Omega$, $n\geq 2$. Then for the following optimization problem
\begin{equation*}
\mu_{\Omega}:=\inf\limits_{u\in D^{1,2}(\Omega)}\frac{\int_{\Omega}|\nabla u|^{2}dx}{\int_{\Omega}V_{*}|u|^{2}dx},
\end{equation*}
we have
\begin{equation}\label{T-4-7}
\frac{(N-2)^{2}}{4n}< \mu_{\Omega}\leq\frac{(N-2)^{2}}{4}.
\end{equation}
\end{corollary}
\textbf{Funding}
 {This work is  supported by
  the National Natural Science Foundation of China (Grant No. 11771395 and 12071431).}

%%%%%%%%%%%%%%%%%%%%%%%%%%%%%%%%%%%[²Î¿¼ÎÄÏ×]%%%%%%%%%%%%%%%%%%%%%%%%%%%%%%%%%%%%%%%%%%%%%%%%%%%%%%

\end{document}